\documentclass[ECP]{ejpecp} 

\usepackage[T1]{fontenc}
\usepackage[utf8]{inputenc}

\usepackage[numeric,initials,nobysame,msc-links,abbrev]{amsrefs}
\usepackage{subcaption}
\usepackage{floatrow}
\usepackage{mathdots}
\usepackage{tikz}
\usepackage[capitalise]{cleveref}

\usetikzlibrary{patterns}

\SHORTTITLE{Sensitive bootstrap percolation second term}

\TITLE{Sensitive bootstrap percolation second term\support{Supported
    by the Austrian Science Fund (FWF): P35428-N and ERC Starting Grant 680275 ``MALIG.''}\
} 

\AUTHORS{%
Ivailo~Hartarsky\footnote{Technische Universit\"at Wien, Institut für Stochastik und Wirtschaftsmathematik, Wiedner Hauptstra\ss e 8-10, A-1040, Vienna, Austria. \EMAIL{ivailo.hartarsky@tuwien.ac.at}}}


\AMSSUBJ{60C05; 60K35} 

\SUBMITTED{March 29, 2023} 
\ACCEPTED{NA} 

\ARXIVID{2303.14910} 

\VOLUME{0}
\YEAR{2020}
\PAPERNUM{0}
\DOI{10.1214/YY-TN}

\ABSTRACT{In modified two-neighbour bootstrap percolation in two dimensions each site of $\mathbb Z^2$ is initially independently infected with probability $p$ and on each discrete time step one additionally infects sites with at least two non-opposite infected neighbours. In this note we establish that for this model the second term in the asymptotics of the infection time $\tau$ unexpectedly scales differently from the classical two-neighbour model, in which arbitrary two infected neighbours are required. More precisely, we show that for modified bootstrap percolation with high probability as $p\to0$ it holds that
\[\tau\le \exp\left(\frac{\pi^2}{6p}-\frac{c\log(1/p)}{\sqrt p}\right)\]
for some positive constant $c$, while the classical model is known to lack the logarithmic factor.}

\newcommand{\cC}{\mathcal{C}}
\newcommand{\cD}{\mathcal{D}}
\newcommand{\cE}{\mathcal{E}}

\newcommand{\cI}{\mathcal{I}}
\newcommand{\cJ}{\mathcal{J}}

\newcommand{\cO}{\mathcal{O}}

\newcommand{\bbP}{{\ensuremath{\mathbb P}} }

\newcommand{\bbZ}{{\ensuremath{\mathbb Z}} }

\begin{document}



\section{Introduction}
\emph{Modified two-neighbour bootstrap percolation} is a monotone cellular automaton on $\bbZ^2$, which may be defined as a growing subset of $\bbZ^2$ as follows. Let $A=A_0\subset\bbZ^2$ be an arbitrary set of initially \emph{infected} sites. For all nonnegative integer $t$ we set
\[A_{t+1}=A_t\cup\left\{x\in\bbZ^2:A_t\cap\{x+e_1,x-e_1\}\neq\varnothing,A_t\cap\{x+e_2,x-e_2\}\neq\varnothing\right\},\]
where $e_1,e_2$ is the canonical basis of $\bbZ^2$. The more \emph{classical two-neighbour bootstrap percolation} is defined similarly, by setting $A'_0=A$ and 
\[A'_{t+1}=A'_t\cup \left\{x\in\bbZ^2:\left|A'_t\cap\{x+e_1,x+e_2,x-e_1,x-e_2\}\right|\ge 2\right\}.\]
Given $A\subset A_t$, one denotes its \emph{closure} by $[A]=\bigcup_{t\ge0}A_t$ and $[A]'=\bigcup_{t\ge 0}A_t'$. One is interested in the behaviour of these models as $A$ is taken at random according to the product Bernoulli measure $\bbP_p$ so that $\bbP_p(x\in A)=p$ for all $x\in\bbZ^2$. It is classical \cite{VanEnter87} that for any $p>0$, $\bbP_p$-almost surely $[A]=[A]'=\bbZ^2$, so it is natural to consider the (random) \emph{infection time}
\begin{align*}\tau&{}=\min\left\{t\ge 0:0\in A_t\right\},&\tau'&{}=\min\left\{t\ge 0:0\in A_t'\right\}.\end{align*}

Traditionally, the modified model is perceived as a technically simpler but morally identically behaved variant of the classical one. As a result, it has been usually treated simultaneously with or before the classical one by the same means. It was proved in \cites{Holroyd03,Gravner08} that for some constant $c>0$
\begin{equation}
\label{eq:GH}
\lim_{p\to 0}\bbP_p\left(\exp\left(\frac{\lambda+o(1)}{p}\right)\le \tau\le \exp\left(\frac{\lambda}{p}-\frac{c}{\sqrt p}\right)\right)=1\end{equation}
with $\lambda=\pi^2/6$ and the same holds for $\tau'$ with $\lambda'=\pi^2/18$. The second order term was important for explaining the \emph{bootstrap percolation paradox}. This is the ongoing phenomenon that simulation-based conjectures on these models were systematically subsequently disproved by rigorous results.

While for the modified model the lower bound in Eq.~\eqref{eq:GH} remains the state of the art, for the classical model it was improved \cite{Hartarsky19}, so that for some $c'>0$
\begin{equation}
\label{eq:HM}\lim_{p\to0}\bbP_p\left(\exp\left(\frac{\lambda'}{p}-\frac{1}{c'\sqrt p}\right)\le \tau'\le \exp\left(\frac{\lambda'}{p}-\frac{c'}{\sqrt p}\right)\right)=1,\end{equation}
making it natural to expect an analogous result for the modified model. Indeed, further studies of simplified so-called local versions of these processes were carried out in \cites{Gravner09,Bringmann12} leading to lower bounds between those of Eqs.~\eqref{eq:GH} and~\eqref{eq:HM} with identical error terms for both processes. Moreover, already Gravner and Holroyd \cite{Gravner08} expended some effort to obtain a decent explicit, albeit non-optimal, value for $c$ in Eq.~\eqref{eq:GH} given by $\sqrt2+o(1)$, since it was simpler to keep track of error terms than for the classical model. 

The only sign of discrepancy between the two models we are aware of arose on the occasion of proving a subsequent improved lower bound for classical bootstrap percolation preceding Eq.~\eqref{eq:HM} and the local ones of \cite{Bringmann12}. Namely, Gravner, Holroyd and Morris \cite{Gravner12} wrote ``Surprisingly, it appears that our proof does
not extend directly to the `modified' bootstrap percolation model'' (also see \cite{Gravner12}), while it did apply to other similar models such as the Frob\"ose and $k$-cross ones \cite{Gravner12}*{Theorem 20}. The subsequent work \cite{Hartarsky19} proving Eq.~\eqref{eq:HM} faced the same difficulty and simply ignored the modified model.

In view of the above, the following result comes as quite a surprise.
\begin{theorem}
\label{th:main}
For the modified two-neighbour bootstrap percolation model in two dimensions there exists $c>0$ such that
\[\lim_{p\to0}\bbP_p\left(\tau\le \exp\left(\frac{\pi^2}{6p}-\frac{c\log(1/p)}{\sqrt p}\right)\right)=1.\]
\end{theorem}
Let us remark that analogous results hold for the critical probability and critical length (see \cite{Gravner08}) or the local version of modified bootstrap percolation (see \cite{Gravner09}) sometimes considered in the literature and they follow in the same way.

It is further good to note that the constant $\lambda$ is known to be the result of a nontrivial optimisation problem \cite{Duminil-Copin23} in some generality including all models mentioned above. Its value is therefore very sensitive to the exact microscopic details of the model. However, the leading order scaling itself is known to admit a simple universal form in much greater generality \cite{Bollobas23}. In view of this and the similarity so far exhibited by models for which information about the second term is available, one might hope for a similar universal form for this corrective term, at least in some generality within the so-called isotropic class. Unfortunately, Theorem~\ref{th:main} strongly indicates that this term is far too sensitive to be captured in general, if even models as similar as the classical and modified ones differ at that level.

\section{Heuristics}
While the proof of Theorem~\ref{th:main} is very simple, we find it important to highlight the somewhat subtle intuition underlying it. The leading order term $\lambda/p$ in Eq.~\eqref{eq:GH} can be understood simply as the limit of $-\frac12\sum_{i=1}^\infty\log (1-(1-p)^i)^2$, which corresponds to the diagonal growth mechanism depicted in Fig.~\ref{fig:diag}. In order to improve it for the classical model, Gravner and Holroyd \cite{Gravner08} considered deviations from the diagonal (see Fig.~\ref{fig:side}) on the critical scale $1/p$ in order to win some entropy. This comes at an energetic price, so the length $1/\sqrt p$ of these deviations was chosen carefully, so that the entropic gain cancels the energetic cost, but not by far, since afterwards entropy itself becomes smaller, deteriorating the result. The same can be done for the modified model to show the upper bound in Eq.~\eqref{eq:GH}.

However, \cite{Gravner08} gives yet another reason for the upper bound in Eq.~\eqref{eq:GH} to hold for the modified model (for the purpose of easily obtaining a good constant). Namely, on scales $1/\sqrt{p}$ (much smaller than the critical one) one can already win an entropy $c/\sqrt p$, using the minimal possible number of infections, but allowing their position to be arbitrary. This hints at what will be the key to Theorem~\ref{th:main}. Namely, on any scale between $1/\sqrt p$ and $1/p$ one can win a second order contribution of $1/\sqrt{p}$, which will accumulate over all scales to give the desired result. 

Finally, let us point out the exact spot witnessing the difference between the modified and classical bootstrap percolation and ultimately leading to the logarithmic correction. When comparing the growth mechanism of Fig.~\ref{fig:side} to Fig.~\ref{fig:diag}, four contributions arise. Firstly, the cost of deviating too much from the diagonal, since the horizontal growth is harder. Secondly, the cost of the infection used to switch direction. Thirdly, the cost of the line free of infections. Fourthly, the fact that the switching infection induces the growth of two lines (a vertical and a horizontal one). The first contribution is going to be subdominant for our purposes (our deviations will be small) and so is the third one (we will consider scales smaller than the critical one, making the absence of infections likely). The second one is simply $p$ per such deviation. It is the fourth contribution that was neglected in \cite{Gravner08} and is crucial to us. Indeed, in both the classical and modified models a single infection can be used to grow two lines. However, in the classical one that is what should be done typically on the critical scale, while for the modified one it is one infection per line. Thus, each line normally costs twice as much for the modified model, so growing two lines with a single direction switching infection is more beneficial. Thanks to this, we can afford to have more such deviation steps, so that together with the entropy contribution we can compensate the cost of these single infections. At this point it only remains to optimise the choice deviation lengths as a function of the scale. We will make no attempt to optimise the constant $c$ in Theorem~\ref{th:main}, but rather simplify the presentation, as the right constant cannot be obtained this way. Yet, we believe Theorem~\ref{th:main} to be sharp up to the value of $c$.

\section{Preliminaries}
For integers $x_1\le x_2$ and $y_1\le y_2$, we define the \emph{rectangle}
\[R(x_1,y_1;x_2,y_2)=\{x_1,\dots,x_2\}\times\{y_1,\dots,y_2\}.\]
We say that a rectangle $R$ is \emph{internally filled}, if $[A\cap R]=R$. For a positive integer $L$ we denote by $\cI(L)$ the event that $R(1,1;L,L)$ is internally filled. For a rectangle $R$ we say that $R$ is \emph{occupied}, if $A\cap R\neq\varnothing$ and denote this event by $\cO(R)$ and its complement by $\cO^c(R)$.

The following is our main goal.
\begin{proposition}[Filling probability]
\label{prop:main}
Let $\lambda=\pi^2/6$. There exist $c>0$ and $p_0>0$ such that for all $p\in(0,p_0)$ and integer $B\ge 2p^{-3/4}$
\[\bbP_p(\cI(B))\ge \exp\left(-2\lambda/p+c\log(1/p)/\sqrt p\right).\]
\end{proposition}
Assuming Proposition~\ref{prop:main}, it is standard to deduce Theorem~\ref{th:main} by the method of \cite{Aizenman88}.
\begin{proof}[Proof of Theorem~\ref{th:main}]
Fix $c>0$ as in Proposition~\ref{prop:main}. By Proposition~\ref{prop:main} with high probability as $p\to 0$ there exist $x_1,x_2\in\bbZ$ with $\max(|x_1|,|x_2|)\le \exp(\lambda/p-c\log(1/p)/(3\sqrt p))$ such that $\cI(R(x_1,x_2;x_1+B,x_2+B))$ occurs with $B=\lceil p^{-3}\rceil$, since these events are independent for disjoint rectangles. Moreover, by the union bound we have that with high probability for all $x_1,x_2$ as above the events $\cO(R(x_1,x_2;x_1+B,x_2))$ and $\cO(R(x_1,x_2;x_1,x_2+B))$ hold. Yet, if both of these high probability events occur, then $\tau\le (B+1)^2+B\exp(\lambda/p-c\log(1/p)/(3\sqrt p))$. To see this, note that the initial square of side length $B$ becomes infected in time at most $(B+1)^2$, while thereafter in each $B$ steps its height and width grow by (at least) 1 towards the origin.
\end{proof}

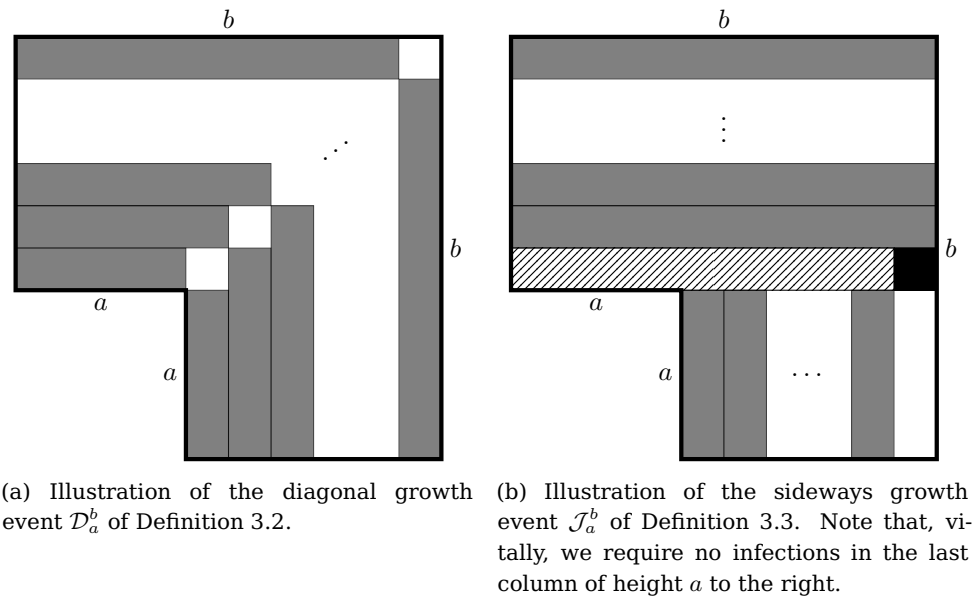
\begin{figure}
\centering
\subcaptionbox{\label{fig:diag}Illustration of the diagonal growth event $\cD_a^b$ of Definition~\ref{def:diagonal}.}
{
        \begin{tikzpicture}[x=0.04\textwidth,y=0.04\textwidth]
        \draw[ultra thick] (4,4)--(4,0)--(10,0)--(10,10)--(0,10)--(0,4)--cycle;
        \foreach \i in {4,5,6,9}
        {
            \draw[fill=black,opacity=0.5] (\i,0) rectangle (\i+1,\i);
            \draw[fill=black,opacity=0.5] (0,\i) rectangle (\i,\i+1);
        }
        \draw (4,2) node[left]{$a$};
        \draw (2,4) node[below]{$a$};
        \draw (10,5) node[right]{$b$};
        \draw (5,10) node[above]{$b$};
        \draw (7.5,7.5) node{$\iddots$};
        \end{tikzpicture}
}%
\quad
\subcaptionbox{\label{fig:side}Illustration of the sideways growth event $\cJ_a^b$ of Definition~\ref{def:sideways}. Note that, vitally, we require no infections in the last column of height $a$ to the right.}
{
    \begin{tikzpicture}[x=0.04\textwidth,y=0.04\textwidth]
        \draw[ultra thick] (4,4)--(4,0)--(10,0)--(10,10)--(0,10)--(0,4)--cycle;
        \draw[fill=black,opacity=0.5] (4,0) rectangle (5,4);
        \draw[fill=black,opacity=0.5] (5,0) rectangle (6,4);
        \draw[fill=black,opacity=0.5] (8,0) rectangle (9,4);
        \draw[fill=black,opacity=0.5] (0,5) rectangle (10,6);
        \draw[fill=black,opacity=0.5] (0,6) rectangle (10,7);
        \draw[fill=black,opacity=0.5] (0,9) rectangle (10,10);
        \draw[pattern=north east lines] (0,4) rectangle (9,5);
        \draw[fill=black] (9,4) rectangle (10,5);
        \draw (4,2) node[left]{$a$};
        \draw (2,4) node[below]{$a$};
        \draw (10,5) node[right]{$b$};
        \draw (5,10) node[above]{$b$};
        \draw (7,2) node{$\dots$};
        \draw (5,8) node{$\vdots$};
        \end{tikzpicture}
        }
    \caption{\label{fig}The two growth mechanisms. Shaded rectangles are required to be occupied; the hatched one is not occupied; the black site is infected.}
\end{figure}

In the remainder of this section we import the relevant parts of \cite{Gravner08}*{Section 3} (see Fig.~\ref{fig}).
\begin{definition}[Diagonal growth]
\label{def:diagonal}
Let $a\le b$ be positive integers. We set
\[\cD_a^b=\bigcap_{i=a+1}^{b}\cO(R(1,i;i-1,i))\cap\cO(R(i,1;i,i-1)).\]
\end{definition}
\begin{definition}[Sideways growth]
\label{def:sideways}
Let $a<b$ be positive integers. We set
\begin{multline*}\cJ_a^b=\bigcap_{i=a+1}^{b-1}\cO(R(i,1;i,a))\cap\cO(R(1,i+1;b,i+1))\\\cap\cO^c(1,a+1;b-1,a+1)\cap \{(b,a+1)\in A\}.\end{multline*}
\end{definition}

\begin{definition}[Alternating growth]
\label{def:alternating}
For positive integers $m$, $(a_i)_{i=1}^{m+1}$ and $(b_i)_{i=1}^{m}$ such that $a_1< b_1\le a_2< b_2\le \dots<b_m\le a_{m+1}$, we further define 
\[\cE(a_1,b_1,\dots,b_m,a_{m+1})=\bigcap_{i=1}^m \cJ_{a_i}^{b_i}\cap\cD_{b_i}^{a_{i+1}}.\]
\end{definition}
Clearly, the events featuring in each of Definitions~\ref{def:diagonal} to~\ref{def:alternating} are independent. Moreover, it is readily checked that $\cI(a)\cap\cD_a^b\subset \cI(b)$ and $\cI(a)\cap\cJ_a^b\subset\cI(b)$ for any positive integers $a< b$ and therefore $\cI(a_1)\cap\cE(a_1,b_1,\dots,b_m,a_{m+1})\subset\cI(a_{m+1})$ for any positive integers $m$ and $a_1<b_1\le\dots <b_m\le a_{m+1}$.

\begin{lemma}
\label{lem:disjoint}
Fix positive integers $a\le b$. Then for different choices of positive integers $a_1< b_1\le\dots<b_m\le a_{m+1}$ with $a_1=a$, $a_{m+1}=b$, the events $\cE(a_1,b_1,\dots,a_{m+1})$ are disjoint.
\end{lemma}
\begin{proof}
Given a realisation $A\in\cE(a_1,b_1,\dots,a_{m+1})$, one can check by induction that for all integer $i\in[1,m]$
\begin{align*}
b_{i}&{}=\min\left\{b'>a_i:(b',a_i+1)\in A\right\},\\
a_{i+1}&{}=\min\left\{a'\ge b_i:A\in\cO^c\left(R\left(1,a'+1;a',a'+1\right)\right)\right\},\end{align*}
with $\min\varnothing=b$, so the sequences are uniquely determined.
\end{proof}

\section{Proof of Proposition~\ref{prop:main}}
In this section we develop our multi-scale strategy for proving Proposition~\ref{prop:main}. Following \cite{Holroyd03}, set $q=-\log(1-p)$ and
\[
f:(0,\infty)\to(0,\infty):z\mapsto-\log(1-e^{-z}).\]
The function $f$ is $\cC^\infty$, decreasing and convex. Set 
\begin{align*}
N&{}=\left\lceil\log (1/p)/(4\log 2)\right\rceil,&
m&{}=1/\left(50\sqrt{q}\right)
\end{align*}
and assume for simplicity that $m$ is an integer. It will be convenient to proceed scale by scale, so let us introduce $\ell^{(n)}=2^n/\sqrt{q}$
for any integer $n\in[0,N]$. We will use the event from Definition~\ref{def:alternating} with sequences $(a_i^{(n)})_{i=1}^{m+1}$ and $(b_i^{(n)})_{i=1}^{m}$ such that \[\ell^{(n)}=a_1^{(n)}<b_1^{(n)}\le a_2^{(n)}<b_2^{(n)}\le\dots<b_{m}^{(n)}\le a_{m+1}^{(n)}=\ell^{(n+1)}.\]
We call such sequences \emph{good}, if for all $n\in[0,N)$ and $i\in[1,m]$ it holds that $b^{(n)}_i-a^{(n)}_i\in[1,2^n]$. It is important to note that, contrary to what would be best for non-modified bootstrap percolation, the number of terms on each scale does not increase with the scale.

Let us first assess the entropy, that is, the number of good sequences.
\begin{lemma}[Entropic gain]
\label{lem:entropy}
There are at least $(12\cdot2^N)^{N(m-1)}$ good sequences.
\end{lemma}
\begin{proof}
Choosing the $a_i^{(n)}$ first so that they differ by at least $2^n$, we get that the number of sequences is at least
\begin{align*}
\prod_{n=0}^{N-1}\binom{\ell^{(n+1)}-\ell^{(n)}-m2^n}{m-1}2^{nm}&{}\ge \prod_{n=0}^{N-1} \frac{(2^{n-1}/\sqrt q)^{m-1}}{m^{m-1}}2^{nm}\\
&{}\ge \prod_{n=0}^{N-1}\left(50\cdot 2^{2n-1}\right)^{m-1}\\&{}= \left(50\cdot2^{N-2}\right)^{N(m-1)}.\qedhere\end{align*}
\end{proof}

We next evaluate the energy cost of sideways growth.
\begin{lemma}[Energy cost]
\label{lem:energy}
Fix positive integers $n\in[0,N)$ and $\ell^{(n)}\le a<b\le \ell^{(n+1)}$ such that $b-a\le 2^n$. Then for any $p<1/2$
\[\frac{\bbP_p(\cJ_a^b)}{\bbP_p(\cD_a^b)}\ge \frac{\exp(-2^{n+2}\sqrt{q})}{2^{2n+3}}.\]
\end{lemma}
\begin{proof}
Definitions~\ref{def:diagonal} and~\ref{def:sideways} give
\begin{align}
\label{eq:PD}\bbP_p\left(\cD_a^b\right)&{}=\exp\left(-2\sum_{i=a}^{b-1}f(iq)\right),\\
\nonumber\bbP_p\left(\cJ_a^b\right)&{}=p\exp\left(-\left(b-a-1\right)\left(f(aq)+f(bq)\right)\right)\exp\left(-q(b-1)\right)\\
\nonumber&{}\ge p\exp\Big(2f\left(q\ell^{(n+1)}\right)-q\ell^{(n+1)}
-2(b-a)\left(f(bq)-(b-a\right)qf'\left(\ell^{(n)}q\right)\Big)\\
\nonumber&{}\ge p\exp\left(2f\left(2^{n+1}\sqrt{q}\right)-2^{n+1}\sqrt{q}-2q(b-a)^2/\left(\ell^{(n)}q\right)\right)\bbP_p\left(\cD_a^b\right)\\
\nonumber&{}\ge p\exp\left(-2\log\left(2^{n+1}\sqrt{q}\right)-2^{n+2}\sqrt{q}\right)\bbP_p\left(\cD_a^b\right)\end{align}
since $f(z)=-\log(1-e^{-z})\ge -\log z$ is decreasing and convex and $f'(z)=1/(1-e^{z})\ge -1/z$. Since $q<2p$ for $p<1/2$, this concludes the proof.
\end{proof}

\begin{proof}[Proof of Proposition~\ref{prop:main}]
Fix $p$ small enough and $B\ge \ell^{(N)}$. 
By Lemmas~\ref{lem:disjoint} to~\ref{lem:energy} we have
\begin{align*}\bbP_p\left(\cI(B)\right)\ge{}& p\bbP_p\left(\cD_{1}^{\ell^{(0)}}\right)\bbP_p\left(\cD_{\ell^{(0)}}^{\ell^{(N)}}\right)\bbP_p\left(\cD_{\ell^{(N)}}^B\right)\left(12\cdot 2^{N}\right)^{N(m-1)}\prod_{n=0}^{N-1}\frac{\exp\left(-m2^{n+2}\sqrt{q}\right)}{2^{m(2n+3)}}\\
\ge{}& \bbP_p\left(\cD_{1}^{\infty}\right)e^{-p^{-1/3}}2^{N(m-1)}\ge\exp\left(2q^{-1}\int_0^\infty f-\frac{\log(1/p)}{300\sqrt{p}}\right),\end{align*}
where we noted that $\bbP_p(\cD_a^b)\bbP_p(\cD_b^d)=\bbP_p(\cD_a^d)$ for any positive integers $a\le b\le d$ by Definition~\ref{def:diagonal} and used Eq.~\eqref{eq:PD}. This concludes the proof, since $\lambda=\int_0^\infty f$ \cite{Holroyd03}.
\end{proof}




\begin{thebibliography}{10}
\bib{Aizenman88}{article}{
      author={Aizenman, M.},
      author={Lebowitz, J.~L.},
       title={Metastability effects in bootstrap percolation},
        date={1988},
        ISSN={0305-4470, 1361-6447},
     journal={J. Phys. A},
      volume={21},
      number={19},
       pages={3801\ndash 3813},
      review={\MR{968311}},
}

\bib{Bollobas23}{article}{
      author={Bollob{\'a}s, B.},
      author={{Duminil-Copin}, H.},
      author={Morris, R.},
      author={Smith, P.},
       title={Universality for two-dimensional critical cellular automata},
        date={2023},
        ISSN={0024-6115},
     journal={Proc. Lond. Math. Soc. (3)},
      volume={126},
      number={2},
       pages={620\ndash 703},
      review={\MR{4550150}},
}

\bib{Bringmann12}{article}{
      author={Bringmann, K.},
      author={Mahlburg, K.},
       title={Improved bounds on metastability thresholds and probabilities for
  generalized bootstrap percolation},
        date={2012},
        ISSN={0002-9947},
     journal={Trans. Amer. Math. Soc.},
      volume={364},
      number={7},
       pages={3829\ndash 3859},
      review={\MR{2901236}},
}

\bib{Duminil-Copin23}{article}{
      author={{Duminil-Copin}, H.},
      author={Hartarsky, I.},
       title={Sharp metastability transition for two-dimensional bootstrap
  percolation with symmetric isotropic threshold rules},
        date={2023},
     journal={arXiv e-prints},
}

\bib{Gravner08}{article}{
      author={Gravner, J.},
      author={Holroyd, A.~E.},
       title={Slow convergence in bootstrap percolation},
        date={2008},
        ISSN={1050-5164},
     journal={Ann. Appl. Probab.},
      volume={18},
      number={3},
       pages={909\ndash 928},
      review={\MR{2418233}},
}

\bib{Gravner09}{article}{
      author={Gravner, J.},
      author={Holroyd, A.~E.},
       title={Local bootstrap percolation},
        date={2009},
        ISSN={1083-6489},
     journal={Electron. J. Probab.},
      volume={14},
       pages={Paper No. 14, 385\ndash 399},
      review={\MR{2480546}},
}

\bib{Gravner12}{article}{
      author={Gravner, J.},
      author={Holroyd, A.~E.},
      author={Morris, R.},
       title={A sharper threshold for bootstrap percolation in two dimensions},
        date={2012},
        ISSN={0178-8051},
     journal={Probab. Theory Related Fields},
      volume={153},
      number={1-2},
       pages={1\ndash 23},
      review={\MR{2925568}},
}

\bib{Hartarsky19}{article}{
      author={Hartarsky, I.},
      author={Morris, R.},
       title={The second term for two-neighbour bootstrap percolation in two
  dimensions},
        date={2019},
        ISSN={0002-9947},
     journal={Trans. Amer. Math. Soc.},
      volume={372},
      number={9},
       pages={6465\ndash 6505},
      review={\MR{4024528}},
}

\bib{Holroyd03}{article}{
      author={Holroyd, A.~E.},
       title={Sharp metastability threshold for two-dimensional bootstrap
  percolation},
        date={2003},
        ISSN={0178-8051},
     journal={Probab. Theory Related Fields},
      volume={125},
      number={2},
       pages={195\ndash 224},
      review={\MR{1961342}},
}

\bib{VanEnter87}{article}{
      author={{van Enter}, A. C.~D.},
       title={Proof of {{Straley}}'s argument for bootstrap percolation},
        date={1987},
        ISSN={0022-4715},
     journal={J. Stat. Phys.},
      volume={48},
      number={3-4},
       pages={943\ndash 945},
      review={\MR{914911}},
}
\end{thebibliography}

\providecommand{\bysame}{\leavevmode\hbox to3em{\hrulefill}\thinspace}
\providecommand{\MR}{\relax\ifhmode\unskip\space\fi MR }
\providecommand{\MRhref}[2]{%
  \href{http://www.ams.org/mathscinet-getitem?mr=#1}{#2}
}
\providecommand{\href}[2]{#2}


\begin{acks}
We thank Augusto Teixeira for several enlightening discussions and IMPA, where this work was initiated, for the hospitality.
\end{acks}


\end{document}